\documentclass[11pt]{amsart}

\usepackage{hyperref}
\usepackage{amsfonts,amssymb,amscd,amsmath,amsthm,enumerate,mathrsfs}
\usepackage{graphicx,type1cm,eso-pic}
\newcommand{\abs}{\vspace*{6pt}}

\DeclareMathOperator{\vol}{vol}

\DeclareMathOperator{\pr}{pr}

\DeclareMathOperator{\id}{id}

\DeclareMathOperator{\const}{const.}

\DeclareMathOperator{\Iso}{Iso}
\DeclareMathOperator{\grad}{grad}

\DeclareMathOperator{\Hess}{Hess}

\DeclareMathOperator{\diag}{diag}

\DeclareMathOperator{\Fix}{Fix}

\DeclareMathOperator{\la}{\langle}
\DeclareMathOperator{\ra}{\rangle}
\newcommand{\h}{h_{\text{top}}}
\newcommand{\skp}{{\langle .,. \rangle}}

\newcommand{\R}{\mathbb{R}}
\newcommand{\Z}{\mathbb{Z}}

\newcommand{\N}{\mathbb{N}}
\newcommand{\C}{\mathbb{C}}

\newcommand{\g}{\mathfrak{g}}

\renewcommand{\g}{\gamma}

\newcommand{\LL}{\mathcal{L}}
\renewcommand{\H}{\mathcal{H}}

\newcommand{\RR}{\mathcal{R}}
\newcommand{\J}{\mathcal{J}}
\renewcommand{\O}{\mathcal{O}}
\newcommand{\M}{\mathcal{M}}
\newcommand{\G}{\mathcal{G}}

\newcommand{\T}{{\mathbb{T}^2}}

\newcommand{\e}{\varepsilon}
\newcommand{\del}{\xi}

\newcommand{\al}{\alpha}
\newcommand{\om}{\omega}
\newcommand{\lam}{\lambda}

\newcommand{\Gam}{\Gamma}

\theoremstyle{plain}
\newtheorem{defn}{Definition}[section]
\newtheorem{lemma}[defn]{Lemma}

\newtheorem{thm}[defn]{Theorem}
\newtheorem{cor}[defn]{Corollary}
\newtheorem{notation}[defn]{Notation}
\newtheorem{mainthm}[defn]{Main Theorem}

\theoremstyle{definition}
\newtheorem{remark}[defn]{Remark}

\newcommand{\mane}{Ma\~n\'e}

\newcommand{\Poincare}{Poincar\'e }

\usepackage{cite}
\def\BibTeX{{\rm B\kern-.05em{\sc i\kern-.025em b}\kern-.08em
    T\kern-.1667em\lower.7ex\hbox{E}\kern-.125emX}}

\usepackage{tikz}
\usetikzlibrary{matrix,shapes,arrows}

\begin{document}

\hypersetup{pdftitle = {Minimal rays on surfaces of genus greater than one -- Part II}, pdfauthor = {Jan Philipp Schr\"oder}}

\title[Minimal rays on surfaces -- Part II]{Minimal rays on surfaces \\ of genus greater than one \\ -- Part II}
\author[J. P. Schr\"oder]{Jan Philipp Schr\"oder}
\address{Faculty of Mathematics \\ Ruhr University \\ 44780 Bochum \\ Germany}
\email{\url{jan.schroeder-a57@rub.de}}
\date{\today}

\begin{abstract}
 We consider any Finsler metric on a closed, orientable surface of genus greater than one. H. M. Morse proved that we can associate an asymptotic direction to minimal rays in the universal cover (in the Poincar\'e disc: a point on the unit circle). We prove here that, if two minimal rays have a common asymptotic direction, which is not a fixed point of the group of deck transformations, then the two rays can intersect at most in a common initial point. This has strong consequences for the structure of the set of minimal geodesics, as well as for the set of Busemann functions associated to the Finsler metric.
\end{abstract}

\maketitle

%\tableofcontents

%%%%%%%%%%%%%%%%%%% 2. ARXIV-PAPER %%%%%%%%%%%%%%%%%%%%%

\section{Introduction and main results}

Let us introduce the relevant objects. Let $M$ be a closed, orientable surface of genus $>1$ and $X$ its universal cover. On $M$, there exists a (hyperbolic) Riemannian metric $g$ of constant curvature $-1$, which can be lifted to $X$, making $(X,g)$ isometric to the \Poincare disc model of the hyperbolic plane, i.e.
\[ X=\{z\in\C : |z|<1\}, \qquad g_z(v,w) = 4\cdot(1-|z|^2)^{-2}\la v,w\ra, \]
writing $\skp,|.|$ for the euclidean inner product and its norm in $\C$. The distance induced by $g$ on $X$ will be denoted by $d_g$. We write $\Gamma\subset\Iso(X,g)$ for the group of deck transformations $\tau:X\to X$ with respect to the covering $X\to M = X/\Gam$, which extend to naturally to $S^1$. Let us write
\[ \Fix(\Gam)=\{\del\in S^1 ~|~ \exists \tau\in\Gam-\{\id\} : \tau\del=\del \}. \]
The $g$-geodesics $\g$ in $X$ are circle segments meeting the ``boundary at infinity'' $S^1=\{z\in\C : |z|=1\}$ orthogonally; the endpoints of $\g$ are denoted by $\g(-\infty), \g(\infty) \in S^1$. Let $\G$ be the set of all oriented, unparametrized $g$-geodesics $\g\subset X$. Recall that any $\tau\in\Gam-\{\id\}$ has exactly two distinct fixed points in $S^1$, which are connected by a unique $g$-geodesic $\g\in \G$ being invariant under $\tau$, $\tau\g=\g$; we will call $\g$ the axis of $\tau$.

We consider any Finsler metric $F:TX\to\R$, which is assumed to be invariant under $\Gamma$ (i.e. being the lift of a Finsler metric on $M$). We write $SX=\{F=1\}\subset TX$ and $c_v:\R\to X$ for the geodesic with respect to $F$ defined by $\dot c_v(0)=v$. We are interested in {\em rays} and {\em minimal geodesics} (with respect to $F$), that is $F$-geodesics $c:[0,\infty)\to X, ~ c:\R\to X$, respectively that, in the universal cover $X$ of $M$, minimize the $F$-length between any of their points. Often, the first step towards understanding the geodesic flow of $F$ is to study the minimal geodesics -- their behavior is to some extend prescribed by the topology of the underlying manifold and hence minimal geodesics form a basic framework for all geodesics; moreover, since the geodesic flow is defined in terms of local minimizers a variational problem, it is natural to first study the global minimizers of the problem.

In 1924, H. M. Morse \cite{morse} studied the global behavior of minimal geodesics on $M$ and proved two theorems, which we recall now.

\begin{thm}[Morse \cite{morse}]\label{morse lemma intro}
 In the notation above, rays and minimal geodesics in $(X,F)$ tend to $\pm\infty$ in tubes around $g$-geodesics $\g\in\G$ of finite width $D$, where $D$ is a constant depending only on $F$ and $g$. In particular, rays $c:[0,\infty)\to X$ have a well-defined endpoint $c(\infty) \in S^1$. 
\end{thm}

For $\del\in S^1$ and $\g\in\G$ set
\begin{align*}
\RR_+(\del) & := \{v \in SX ~|~ c_v:[0,\infty)\to X \text{ is a ray, } c(\infty)=\del \}, \\
\M(\g) & := \{v \in SX ~|~ c_v:\R\to X \text{ is a minimal geodesic, } c(\pm\infty)=\g(\pm\infty) \}.
\end{align*}
Hence, any $F$-ray in $X$ belongs to some $\RR_+(\del)$, while any $F$-minimal geodesic belongs to some $\M(\g)$. Conversely, we have $\M(\g)\neq \emptyset$ for all $\g\in \G$ and $\RR_+(\del)$ projects to all of $X$ for all $\del\in S^1$.

The second theorem clarifies the structure of some of the sets $\RR_+(\del)$ for fixed $\del\in S^1$. If $\tau\in\Gam-\{\id\}$ has as its axis some $\g\in \G$ and if $\del=\g(\infty) \in \Fix(\Gam)$ is the corresponding fixed point of $\tau$, then the structure of $\RR_+(\del)$ can be described as follows, cf. Figure \ref{fig_periodic}. Let us write
\[ \M_{per}(\g) := \{ v\in \M(\g) : \tau c_v(\R)=c_v(\R) \}. \]

\begin{thm}[Morse \cite{morse}]\label{morse periodic}
 If $\tau\in \Gam$ and if $\g\in\G$ is the axis of $\tau$ with endpoint $\del := \g(\infty) \in \Fix(\Gam)$, then $\M_{per}(\g)\neq \emptyset$. Moreover, no ray from $\RR_+(\del)-\M_{per}(\g)$ can intersect any minimal geodesic from $\M_{per}(\g)$ and every ray in $\RR_+(\del)$ is asymptotic near $+\infty$ to some minimal geodesic from $\M_{per}(\g)$. All minimal geodesics in $\M(\g)-\M_{per}(\g)$ are heteroclinic between a pair of neighboring periodic minimal geodesics in $\M_{per}(\g)$.
\end{thm}

Note that in the last statement of Theorem \ref{morse periodic}, we implicitly exclude any minimal geodesic in $\M(\g)$ to be homoclinic to a single periodic minimal geodesic in $\M_{per}(\g)$. In particular, if $\M_{per}(\g)$ consists of only one geodesic, then $\M(\g)=\M_{per}(\g)$.

\begin{figure}%[htb!]
\centering
\includegraphics[scale=0.4]{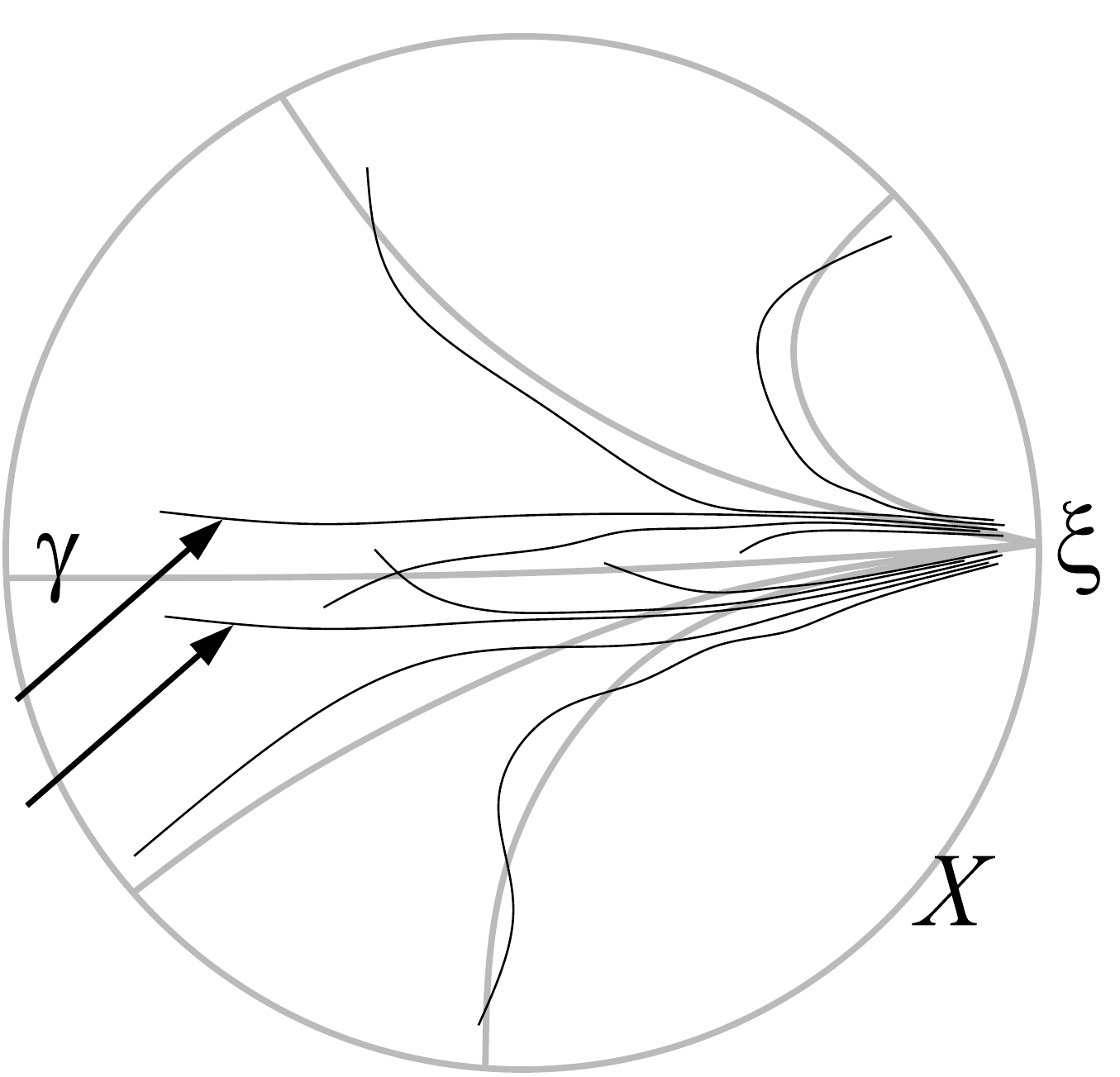}
\caption{The structure of $\RR_+(\del)$ in $X\subset \C$ in Theorem \ref{morse periodic} for $\del\in \Fix(\Gam)\subset S^1$. Minimal rays with respect to $F$ are depicted in black, their corresponding $g$-geodesics in gray. $\g$ is the axis of $\tau$ and the two arrows show a pair of $\tau$-invariant minimal geodesics in $\M_{per}(\g)$. \label{fig_periodic}}
\end{figure}

\abs

Let us remark that in 1932, G. A. Hedlund \cite{hedlund} proved results analogous to Theorems \ref{morse lemma intro} and \ref{morse periodic} for the case of a genus $1$ surface, i.e. the 2-torus $\T=\R^2/\Z^2$. Here the minimal geodesics and rays move along straight euclidean lines in the universal cover $X=\R^2$ and hence also here we can define the sets $\RR_+(\del)$ with $\del\in S^1$, $c(\infty)=\del$ being the direction of the accompanying line. Then for $\del\in S^1$ with rational slope -- that is ``fixed points'' of $\Z^2$ -- the set $\RR_+(\del)$ has the analogous structure as in the case of genus $>1$ in Theorem \ref{morse periodic}, the difference being that, as a set, $\M_{per}(\g)$ is $\Z^2$-invariant, while for genus $>1$ it is only $\tau$-invariant.

Later, in 1988, V. Bangert \cite{bangert} proved moreover, that the sets $\RR_+(\del)$ with $\del\in S^1$ of irrational slope have a rather simple structure: the rays in $\RR_+(\del)$ can intersect at most in common initial points and the set of minimal geodesic with a fixed irrational slope contains no intersecting minimal geodesics in $\R^2$ (in fact, not even in $\T$).

\abs

The results of Hedlund and Bangert together give a complete picture of the structure of all rays in the case of the torus. Returning to the case of genus $>1$, what is left open in understanding the structure of all rays is the structure of the sets $\RR_+(\del)$ with $\del\notin\Fix(\Gam)$. Building on our work from \cite{paper1}, we will prove here the following theorem.

\begin{mainthm}\label{main thm}
 If $\del\in S^1-\Fix(\Gam)$, then
 \[ \liminf_{t\to\infty} d_g(c_v(\R),c_w(t)) = 0  \qquad \forall v,w\in \RR_+(\del). \]
\end{mainthm}

Hence, any two rays in $\RR_+(\del)$ with $\del\notin \Fix(\Gam)$ will approach each other at some points near $+\infty$.

\begin{remark}
 It was previously known \cite{coudene-schapira}, that if $F$ is a Riemannian metric on $M$ with non-positive curvature $K\leq 0$, then any flat strip in $X$, i.e. an isometrically embedded euclidean strip
 \[ (\R\times [0,a], \skp_{euc}) \hookrightarrow (X,F), \qquad a > 0 , \]
 is periodic, i.e. bounded by two axes of some non-trivial $\tau\in \Gam-\{\id\}$. All other flat strips in such a non-positively curved surface have to be ``thin'' in the ends near $\pm\infty$. Main Theorem \ref{main thm} is a generalization of this fact to general Finsler metrics on $M$.
\end{remark}

Main Theorem \ref{main thm} has a series of corollaries.

\begin{cor}\label{cor no intersections}
 If $\del\in S^1-\Fix(\Gam)$ and if $v,w\in \RR_+(\del)$ with $c_w(0)=c_v(a)$ for some $a>0$, then $c_w(t)=c_v(t+a)$.
\end{cor}

Hence two rays with a common asymptotic direction not in $\Fix(\Gam)$ can intersect at most in a common initial point, which is analogous to V. Bangert's result in the case of the 2-torus.

One of our original motivations to prove Corollary \ref{cor no intersections} was to obtain the following theorem. We write $\pr(\M)$ for the projection of the set $\M = \bigcup\{\M(\g) : \g\in \G \}$ of all minimal geodesics into the $F$-unit tangent bundle of $M$, $B_F(x,r)\subset X$ is the $F$-ball of radius $r$, centered at any $x$ in the universal cover $X$. Let us denote the geodesic flow of $F$ by $\phi_F^t:SX\to SX$.

\begin{thm}[cf. \cite{GKOS}]\label{thm vol-entropy}
 If $M$ is a closed surface and $g$ is any Riemannian metric on $M$, then
 \[ \h(\phi_g^t|_{\pr(\M)}) = h(g) := \lim_{r \to \infty }\frac{1}{r}\log \vol_g B_g(x,r) . \]
\end{thm}

Here, $\h(\phi_g^t|_{\pr(\M)})$ denotes the topological entropy of the restricted geodesic flow $\phi_g^t:\pr(\M)\to\pr(\M)$; $h(g)$ is often called volume growth of $(M,g)$. Note that, for general closed Riemannian manifolds, one has the following estimate, cf. Theorem 9.6.7 in \cite{KH}:
\[ \h(\phi_g^t|_{\pr(\M)}) \geq h(g) . \]
Although Theorem \ref{thm vol-entropy} was proved by slightly different means, one can derive the following strengthening of Theorem \ref{thm vol-entropy}, using Corollary \ref{cor no intersections}.

\begin{cor}\label{cor local entropy}
 For all $\g\in \G$, the topological entropy of $\phi_F^t:\M(\g)\to\M(\g)$ vanishes, i.e. $\h(\phi_F^t|_{\M(\g)}) = 0$.
\end{cor}

To describe the structure of $\RR_+(\del)$ further, write
\[ d_F(x,y) := \inf\{\textstyle \int_0^1 F(\dot c) dt ~|~ c:[0,1]\to X ~ C^1, c(0)=x,c(1)=y \} \]
for the $F$-distance between $x,y\in X$. For a sequence $\{x_n\}\subset X$ and $\del\in S^1$ we shall write $x_n\to\del$, if this is true in the euclidean metric in $\C\supset X \cup S^1$. Fixing any ``origin'' $o\in X$, the set of {\em horofunctions $\H_+(\del)$ of direction $\del$} is the set of all possible $C_{loc}^0$ limit functions $u:X\to \R$ of sequences
\[ u_n(x) = d_F(o,x_n)-d_F(x,x_n), \qquad \text{ where }x_n \to \del. \]

\begin{cor}\label{cor unique busemann}
 If $\del\in S^1-\Fix(\Gam)$, then the set $\H_+(\del)$ consists of precisely one horofunction. Hence, for all sequences $x_n\to\del$, the sequence of functions $d_F(o,x_n)-d_F(.,x_n)$ converges in $C_{loc}^0$ to the unique function $u\in \H_+(\del)$.
\end{cor}

One can describe the set $\RR_+(\del)$ in terms of the horofunctions in $\H_+(\del)$. From Corollary \ref{cor unique busemann}, we obtain the following strengthening of Corollary \ref{cor no intersections}.

\begin{cor}\label{R Lipschitz graph}
 If $\del\in S^1-\Fix(\Gam)$, then for all $\e>0$ the set $\phi_F^\e(\RR_+(\del))\subset SX$ is locally a Lipschitz graph over its projection in $X$.
\end{cor}

Next, let us observe that most sets $\M(\g)$ are not very big.

\begin{cor}\label{cor unique min geod}
 For all $\del\in S^1$, there exists a countable set $A\subset S^1$, such that $\M(\g)$ consists of only one minimal geodesic for all $\g\in \G$ connecting a point in $S^1-A$ to $\del$.
\end{cor}

Given some direction $\del_+\in X$, then any point $x\in X$ determines a unique $\del_-\in S^1$ as follows.

\begin{cor}\label{cor backwards dir}
 If $x\in X$ and $\del\in S^1$, then there exists a unique $\g\in \G$ with $\g(\infty)=\del$, such that $x$ lies in the closed strip between a pair of (not necessarily distinct) minimal geodesics from $\M(\g)$.
\end{cor}

We proved that certain behavior of minimal geodesics (in particular: intersecting minimal geo\-de\-sics) in $\M(\g)$ occurs only for axes $\g\in \G$. As it turns out, the existence of such behavior is very exceptional. To be more precise, let us recall the following special case of the main result from \cite{paper_generic}, writing
\[ E := \{ f:M\to \R ~|~ f \text{ is } C^\infty, ~ f(x)>0 ~ \forall x \in M \} . \]

\begin{thm}[cf. \cite{paper_generic}]\label{thm generic}
 Given any Finsler metric $F$ on the surface $M$, there exists a residual subset $\O_F\subset E$ (i.e. a countable intersection of open and dense sets) in the $C^\infty$ topology, such that for the conformally perturbed Finsler metrics $f\cdot F$ with $f\in \O_F$, in any free homotopy class $M$ there exists precisely one shortest closed geodesic with respect to $f\cdot F$.
\end{thm}

Hence, for the {\em conformally generic} Finsler metrics, that is Finsler metrics $f\cdot F$ with $f\in \O_F$, the sets $\M_{per}(\g)$ in Theorem \ref{morse periodic} consist of only one minimal geodesic. We obtain the following corollary.

\begin{cor}\label{cor generic}
 If $F$ is conformally generic, then the statements of Main Theorem \ref{main thm} and of Corollaries \ref{cor no intersections}, \ref{cor unique busemann} and \ref{R Lipschitz graph} are true for all $\del\in S^1$.
\end{cor}

 Note, that in the case of the 2-torus, the $\Z^2$-invariance of the sets $\M_{per}(\g)$ discussed above leads to a quite different picture for generic Finsler metrics: Here, generic Finsler metrics admit {\em many} horofunctions and intersecting minimal geodesics for {\em every} rational direction $\del\in S^1$.

We close the introduction with two remarks.

\begin{remark}
 Differently from what happens on the torus (recall the examples of Riemannian metrics on $\mathbb{T}^3$ due to G. A. Hedlund \cite{hedlund}), some of our techniques work also in higher dimensions. In particular, Theorem \ref{morse lemma intro} continues to hold \cite{klingenberg}. Here the setting is a closed Riemannian manifold $(M,g)$, such that $g$ has everywhere negative sectional curvatures, and then considering an arbitrary Finsler metric $F$ on $M$. This, however, is a topic for future research.
 
 Another way of generalizing our results is to consider non-compact surfaces. It seems very likely to us that under suitable conditions, our main results and arguments continue to work without much alteration, e.g., for Finsler metrics on complete hyperbolic surfaces of finite volume such as the once-punctured torus.
\end{remark}

\begin{remark}
 Another novelty in this paper, compared to the work of Morse, is the use of Finsler instead of Riemannian metrics. It was observed by E. M. Zaustinsky \cite{zaustinsky}, that the results of Morse carry over to these much more general systems. Moreover, it is known that Finsler metrics can be used to describe the dynamics of arbitrary Tonelli Lagrangian systems in high energy levels, cf. \cite{contreras1}. Let us remark that also in the torus case, all results carry over to general Finsler metrics (hence Tonelli Lagrangian systems), cf. \cite{zaustinsky} and \cite{paper1_diss}.
\end{remark}

{\bf Structure of this paper.} 
This paper continues earlier work \cite{paper1}, where we began to study the dynamics of minimal rays of a Finsler metric on a closed orientable surface of genus at least two. Hence, we refer to \cite{paper1} for some proofs. In Section \ref{section basics} we recall known facts on rays and minimal geodesics in surfaces of higher genus, including some material on horofunctions. Section \ref{section main} is devoted to the proof of Main Theorem \ref{main thm}, with Subsection \ref{subsection cor} containing the proofs of its corollaries.

\section{Basics on minimal geodesics in surfaces of higher genus}\label{section basics}

We write $\pi:TX\to X$ for the canonical projection, $0_X$ denotes the zero section and $T_xX=\pi^{-1}(x)$ the fibers. The norm and distance of $g$ on $X$ are denoted by $|.|_g, d_g$, respectively. Let us recall the definition of a Finsler metric and refer to \cite{BCS} for basics on Finsler geometry. The reader unfamiliar with Finsler geometry can without much loss think of a Finsler metric as the norm coming from some Riemannian metric.

\begin{defn}\label{def finsler}
A function $F:TX\to [0,\infty)$ is a \emph{Finsler metric on $X$}, if the following conditions are satisfied:
\begin{enumerate}
\item (smoothness) $F$ is $C^\infty$ in $TX-0_X$,

\item (positive homogeneity) $F(\lam v) = \lam F(v)$ for all $v\in TX, \lam\geq 0$,

\item (strict convexity) the fiberwise Hessian $\Hess(F^2|_{T_xX})$ of the square $F^2$ is positive definite at every vector $v\in T_xX-\{0\}$, for all $x\in X$.
\end{enumerate}
\end{defn}

We will for the rest of the paper fix the Finsler metric $F$ and assume that $F$ is $\Gam$-invariant, meaning that the $g$-isometries $\tau\in \Gam$ are also $F$-isometries:
\[ F(d\tau(v))=F(v) \qquad \forall v\in TX, \tau\in\Gamma. \]
As a consequence of the compactness of $M$, the Finsler metric $F$ is uniformly equivalent to the norm of $g$, i.e. there exists a constant $c_F>0$, such that
\[ \frac{1}{c_F} \cdot F \leq |.|_g \leq c_F\cdot F . \]

\begin{notation}
 We will often drop the dependence of $F$ of objects that are defined with respect to $F$; e.g. geodesics will refer to $F$-geodesics. We will always denote $g$-geodesics by $\g,\g_n$ etc., while $F$-geodesics will be termed $c,c_n$ etc..
\end{notation}

We write $SX=\{v\in TX : F(v)=1\}$ for the unit tangent bundle of $F$. The geodesic flow $\phi_F^t:SX\to SX$ of $F$ is given by $\phi_F^tv=\dot c_v(t)$, where $c_v:\R\to X$ is the $F$-geodesic with initial velocity $\dot c_v(0)=v$. We write
\[ l_F(c) = \int_a^bF(\dot c) dt \]
for the $F$-length of absolutely continuous ($C^{ac}$) curves $c:[a,b]\to X$ and
\[ d_F(x,y) = \inf\{ l_F(c) ~|~ c:[0,1]\to X ~ C^{ac}, ~ c(0)=x, ~ c(1)=y \} \]
for the $F$-distance. Note that if $F$ is not reversible, i.e. if not $F(\lam v)=|\lam|F(v)$ for all $\lam\in\R$, we have $d_F(x,y)\neq d_F(y,x)$ in general.

\begin{defn}
A $C^{ac}$ curve segment $c:[a,b]\to X$ is said to be \emph{minimal}, if $l_F(c)=d_F(c(a),c(b))$. Curves $c:[0,\infty)\to X$, $c:(-\infty,0]\to X$, $c:\R\to X$ are called \emph{forward rays, backward rays, minimal geodesics}, respectively, if each restriction $c|_{[a,b]}$ is a minimal segment. Set
\begin{align*}
\RR_- & := \{v \in SX : c_v:(-\infty,0]\to X \text{ is a backward ray} \}, \\
\RR_+ & := \{v \in SX : c_v:[0,\infty)\to X \text{ is a forward ray} \}, \\
\M & := \{v \in SX : c_v:\R\to X \text{ is a minimal geodesic} \}.
\end{align*}
\end{defn}

We will in this paper mainly be concerned with forward rays, the results for backward rays being completely analogous. The sets $\RR_-,\RR_+$ and $\M$ are $\phi_F^t$-invariant for $t<0,t>0, t\in\R$, respectively. Moreover, all sets in the above definition are closed subsets of $SX$ and (when seen in the quotient $SM$) the $\al$-, $\om$-limit sets of $\RR_-,\RR_+$, respectively, are contained in $\M$. 

The following lemma is a key property of rays. It excludes in particular successive intersections of rays and shows that asymptotic rays can cross only in a common initial point. The idea of the proof is classical and can be found e.g. in \cite{paper1_diss}, Lemma 2.20.

\begin{lemma}\label{crossing minimals}
Let $v,w\in \RR_+$ with $\pi w=c_v(a)$ for some $a>0$, but $w\neq \dot c_v(a)$. Then for all $\delta>0$
\[ \inf \{ d_g(c_v(s),c_w(t)) : s\in [a,\infty), t\in [\delta,\infty) \} > 0. \]
\end{lemma}

\subsection{Asymptotic directions of minimal rays}\label{section morse}

The following theorem due to H. M. Morse, which we already loosely stated as Theorem \ref{morse lemma intro} in the introduction and call the {\em Morse Lemma}, is the starting point for our work. The fact that the Morse Lemma also holds in the Finsler case was first observed by E. M. Zaustinsky \cite{zaustinsky}. From now on we use the fact that $M$ carries a Riemannian metric $g$ of strictly negative curvature.

\begin{thm}[Morse Lemma \cite{morse}]\label{morse lemma}
There exists a constant $D\geq 0$ depending only on $F, g$ with the following property: For any two points $x,y\in X$, any $F$-minimal segment $c:[0,1]\to X$ with $c(0)=x, ~ c(1)=y$ satisfies
\[ \max_{t\in[0,1]} d_g(\g_{x,y},c(t)) \leq D,  \]
where $\g_{x,y}\subset X$ is the unique $g$-geodesic segment from $x$ to $y$.
\end{thm}

As in the introduction, we write $\G$ for the set of all oriented, unparame\-trized $g$-geodesics $\g\subset X$. We can think of $\G$ as $S^1\times S^1-\diag$, associating to $\g\in\G$ its pair of endpoints $(\g(-\infty),\g(\infty))$ on $S^1$, where
\[ \g(\pm\infty)=\lim_{t\to\pm\infty} \g(t) \qquad \text{ in the euclidean sense in $\C \supset X$}. \]

If $c:[0,\infty)\to X$ is a ray (with respect to $F$) and $T_n\to \infty$, we let $\g_n$ be the sequence of $g$-geodesic segments from $c(0)$ to $c(T_n)$. Then $\g_n$ converges to a unique $g$-ray $\g:[0,\infty)\to X$ (convergence in the sense of initial velocities), independently of the choice of $T_n$, and we write $c(\infty) := \g(\infty)$. The uniqueness follows easily from the Morse Lemma and the fact that different $g$-rays initiating from $c(0)$ diverge in $X$ due to negative curvature of $g$. Moreover, if $\g:\R\to X$ is a $g$-geodesic, we find a convergent subsequence of $F$-minimal segments from $\g(-n)$ to $\g(n)$, which yields a minimal geodesic $c:\R\to X$ with $c(\pm\infty)=\g(\pm\infty)$. In this way we see that the Morse Lemma holds equally well for points $x \neq y$ in $X \cup S^1$. Furthermore, it is easy to show for $v_n,v\in\RR_\pm$, that
\[ v_n\to v \text{ in } SX \quad \implies  \quad  c_{v_n}(\pm\infty)\to c_v(\pm\infty) \text{ in } S^1. \]

\begin{defn}
For $\del  \in S^1$ and $\g\in \G$ we set
\begin{align*}
 \RR_\pm(\del ) & := \{ v\in \RR_\pm : c_v(\pm\infty)=\del \}, \\
 \M(\g) & := \{ v\in \M : c_v(-\infty)=\g(-\infty) \text{ and } c_v(\infty)=\g(\infty) \}.
\end{align*}
\end{defn}

Hence, by the discussion above, we have
\[ \RR_\pm = \bigcup \big\{ \RR_\pm(\del) : \del\in S^1 \}, \qquad \M = \bigcup\big\{ \M(\g) : \g\in\G \big\} .\]

As observed already by Morse \cite{morse}, in each $\M(\g)$ there are two particular minimal geodesics, which we call the {\em bounding geodesics} of $\M(\g)$. We give the proof here to indicate that it holds equally well for general Finsler metrics.

\begin{lemma}[bounding geodesics]\label{bounding geodesics morse}
For all $\g\in\G$, there are two particular, non-intersecting, minimal geodesics $c_\g^0,c_\g^1$ in $\M(\g)$, such that all minimal geodesics in $\M(\g)$ lie in the strip in $X$ bounded by $c_\g^0(\R)$ and $c_\g^1(\R)$.

Moreover, if $S\subset X$ is the closed strip between $c_\g^0,c_\g^1$, then any ray $c:[0,\infty)\to X$ initiating in $S$ with $c(\infty)=\g(\infty)$ lies eventually in $S$, i.e. there exists $T \geq 0$ with $c[T,\infty)\subset S$. The analogous statement holds for backward rays with $c(-\infty)=\g(-\infty)$.
\end{lemma}

As a rule, we will always assume that $c_\g^1$ lies left of $c_\g^0$ with respect to the orientation of $\g$.

\begin{proof}
Let $\g_n\subset X$ be a sequence of $g$-geodesics with $\overline \g_n\cap \overline \g = \emptyset$ in $X\cup S^1$ (i.e. also the endpoints on $S^1$ are disjoint) and such that $\g_n(\pm\infty)\to\g(\pm\infty)$. Let us assume that $\g_n$ lie left of $\g$. If $c_n$ is a sequence of minimal geodesics in $\M(\g_n)$, then by minimality no $c_n$ can intersect any minimal geodesic from $\M(\g)$ and hence, we find a unique limit minimal geodesic $c_\g^1$ in $\M(\g)$. Analogously one obtains $c_\g^0$. By construction, $c_\g^i$ do not intersect, since their defining sequences do not intersect.

Suppose $c:[0,\infty)\to X$ is a ray with e.g. $c(0)\in S$ and $c(\infty)=\g(\infty)$, but $c[T,\infty)\not\subset S$ for all $T\geq 0$. We then find $t_n\to\infty$ with $c(t_n)\notin S$, e.g. lying left of $c_\g^1$, while $c(t_n)\to\g(\infty)$. Thus, if $c_n\to c_\g^1$ as in the construction of $c_\g^1$, we obtain by the asymptotic behavior successive intersections of $c_n$ with $c$ for large $n$, contradicting minimality.
\end{proof}

\subsection{Horofunctions}\label{section weak KAM}

We introduce in this subsection a useful class of functions $u:X\to \R$, which is naturally associated to the Finsler metric $F$. It has been studied in various situations in the literature, in particular in the setting of Hadamard manifolds. Let us fix any ``origin'' $o\in X$. For a sequence $x_n\in X$ with $x_n\to \del \in S^1$ (in the euclidean sense), any $C_{loc}^0$ limit $u\in C^0(X)$ of the sequence of functions
\[ x \mapsto d_F(o,x_n) - d_F(x,x_n)  \]
is called a {\em horofunction of direction $\del\in S^1$}. The set of all horofunctions of direction $\del$ will be denoted by
\[ \H_+(\del) := \{ u\in C^0(X) ~|~ \exists x_n\to\del : u = \lim d_F(o,x_n) - d_F(.,x_n) \} \]
and we set
\[ \H_+ := \bigcup\big\{ \H_+(\del) : \del\in S^1 \big\} . \]

\begin{lemma}\label{existence horofunction}
If $x_n\in X$ with $x_n\to \del\in S^1$, then the sequence of functions $d_F(o,x_n) - d_F(. ,x_n) : X\to \R$ has $C_{loc}^0$ limit functions $u\in C^0(X)$. Moreover, any so obtained function $u\in \H_+(\del)$ has the following two properties:
\begin{enumerate}
 \item $u(y)-u(x)\leq d_F(x,y)$ for all $x,y\in X$,

 \item for all $x\in X$ there exists a forward ray $c:[0,\infty)\to X$ with $c(0)=x$, $c(\infty)=\del$, $F(\dot c)=1$ and $u(c(t))-u(x) = t$ for all $t\in[0,\infty)$.
\end{enumerate}
\end{lemma}

The proof of Lemma \ref{existence horofunction} can be found as Lemma 3.5 in \cite{paper1}.

For $u\in \H_+$ we will write
\[ \J_+(u) := \{ v\in SX : u \circ c_v(t)-u \circ c_v(0) = t ~ \forall t\geq 0 \}. \]
It follows from property (1) in Lemma \ref{existence horofunction}, that $\J_+(u)\subset \RR_+$; property (2) shows $\pi(\J_+(u)) = X$. Moreover, property (1) and the uniform equivalence of $F$ and $g$ show that $u\in\H_+$ is Lipschitz with respect to $d_g$,
\[ |u(x)-u(y)| \leq c_F\cdot d_g(x,y) \qquad \forall u\in \H_+, x,y\in X. \]

The following lemma is well-known; indications to the proof can be found in Lemma 3.2 in \cite{paper1}. If $u:X\to\R$ is differentiable in $x\in X$, write
\[ \grad_F u(x) := \LL_F^{-1}(du(x)), \quad \text{where } \LL_F(v) = \frac{1}{2} d_vF^2(v) \in T_{\pi v}^*X \subset T^*X. \]
Here $d_v$ denotes differentiation along the fiber, $\LL_F:TX\to T^*X$ is the Legendre transform associated to $F$. Note that if $F$ is Riemannian, then $\grad_Fu$ is the usual gradient.

\begin{lemma}\label{lemma fathi}
Let $u\in \H_+$. Then for all $\e>0$, $u$ is differentiable in $\pi\phi_F^\e(\J_+(u))$. If $u$ is differentiable in $x\in X$, then
\[ \J_+(u)\cap T_xX = \{ \grad_Fu(x) \}. \]
Moreover, the set $\phi_F^\e(\J_+(u))$ for $\e>0$ is locally a Lipschitz graph over its projection via $\pi$ in $X$.
\end{lemma}

We have two corollaries from Lemma \ref{lemma fathi}.

\begin{cor}\label{as dir well-def}
If $u\in\H_+(\del)$, then $\J_+(u)\subset \RR_+(\del)$. Moreover,
\[ \RR_+(\del) = \bigcup \big\{ \J_+(u) : u \in \H_+(\del) \big\}. \]
\end{cor}

\begin{proof}
Let $v\in \J_+(u)$ and $t>0$. By property (2) in Lemma \ref{existence horofunction}, there exists a ray $c:[0,\infty)\to X$ with $\dot c\in\J_+(u)$ and $c(0)=c_v(t), c(\infty)=\del$. By Lemma \ref{lemma fathi} we find $\dot c(0)=\dot c_v(t)$, i.e. also $c_v(\infty)=c(\infty)=\del$ and $\J_+(u)\subset \RR_+(\del)$. For the second claim, let $v\in \RR_+(\del)$ and $x_n := c_v(n)$. The corresponding horofunction (called the Busemann function of $c_v$) belongs to $\H_+(\del)$ and one can easily see that $u \circ c_v(t)-u \circ c_v(0) = t$ for $t\geq 0$ by minimality of $c_v$, i.e. $v\in\J_+(u)$.
\end{proof}

\begin{cor}\label{J=J' then u=u'}
If $u,u'\in\H_+$ with $\J_+(u)=\J_+(u')$, then $u=u'$.
\end{cor}

\begin{proof}
Let $U\subset X$ be the set where both $u,u'$ are differentiable. $U$ has full measure by Rademacher's Theorem (horofunctions are Lipschitz). Lemma \ref{lemma fathi} and $\J_+(u)=\J_+(u')$ show $du(x)=du'(x)$ for all $x\in U$. Hence $u-u'$ is constant, while $u(o)=0=u'(o)$.
\end{proof}

The assumption $\dim X=2$ can be used to find special horofunctions in $\H_+(\del)$, as seen in the next Lemma. For its proof we refer to Appendix A of \cite{paper1}.

\begin{lemma}[bounding horofunctions] \label{bounding horofunctions}
For fixed $\del\in S^1$, there exist two unique $u_0,u_1\in \H_+(\del)$ with the following property: for all sequences $u_n\in \H_+(\del_n)$ with $\del_n\to \del$ and $\del_n\neq \del$, any $C_{loc}^0$ limit lies in $\{u_0,u_1\}$. More precisely, assuming the counterclockwise orientation of $S^1$, we have
\[ \lim_{n\to\infty}u_n = \begin{cases} u_0 & : \text{ if }~ \del_n < \del ~ \forall n \\ u_1 & : \text{ if }~ \del_n > \del ~ \forall n \end{cases}. \]
\end{lemma}

We will use Lemma \ref{bounding horofunctions} in order to obtain ``recurrent horofunctions''. Namely, for $\tau\in\Gam$ and $u\in \H_+(\del)$ define
\[ (\tau u)(x) := u\circ \tau^{-1}(x)- u\circ \tau^{-1}(o). \]
Then $\tau u$ is the horofunction with
\[ \J_+(\tau u) = \{ d\tau(v) : v\in \J_+(u) \}, \]
and in particular $\tau u\in \H_+(\tau\del)$. Now, if $\{\tau_k\}\subset\Gam$ is a sequence of deck transformations, such that $\tau_k\del\to\del$ and $\del < \tau_k\del$ for all $k$ in the counterclockwise orientation of $S^1$, then Lemma \ref{bounding horofunctions} shows
\[ \tau_k u_1 \to u_1 \qquad \text{ in } C_{loc}^0. \]

\section{The proof of the Main Theorem}\label{section main}

We define the following numbers, characterizing the ``width'' of asymptotic directions in $X$.

\begin{defn}\label{def width}
For $\del\in S^1$ and $\g\in\G$ set
\begin{align*}
w_\pm(\del) &:= \sup \big\{ \liminf_{t\to \pm\infty} d_g(c_v(\R),c_w(t)) : v,w\in\RR_\pm(\del) \big\} , \\
w_0(\g) &:= \sup \big\{ \inf_{t\in\R} d_g(c_v(\R),c_w(t))  : v,w\in\M(\g) \big\} .
\end{align*}
\end{defn}

\begin{remark}
 It is easy to see that, if $\g\in \G$ and if $c_\g^0,c_\g^1$ are the two bounding geodesics of $\M(\g)$ given by Lemma \ref{bounding geodesics morse}, then
 \[ w_0(\g) = \inf_{t\in\R} d_g(c_\g^0(\R),c_\g^1(t)) . \]
\end{remark}

In order to make use of the width, we have to connect the behavior of $w_+$ and $w_0$ to the group $\Gam$. In the following definition we ``lift'' several dynamical notions of $g$-geodesics in the compact quotient $M$ to the covering $X$, expressing them in terms of $\G$ and $\Gam$.

\begin{defn}\label{def positive gamma}
A $g$-geodesic $\g\in\G$ is called an {\em axis}, if there exists some $\tau\in \Gam-\{\id\}$ with $\tau\g=\g$.

A sequence $\{\g_k\}\subset\G$ {\em converges} to $\g\in\G$, if for the pairs of endpoints $\g_k(-\infty) \to \g(-\infty)$ and $\g_k(\infty)\to \g(\infty)$ in $S^1$.

A sequence $\{\tau_k\}\subset\Gamma$ is \emph{positive for $\g\in\G$}, if there exists a sequence $\{x_k\}\subset \g$ with $x_k\to \g(\infty)\in S^1$ (with respect to the topology of $\C\supset X$) and a compact set $K\subset X$, such that $\tau_kx_k\in K$ for all $k$.

A $g$-geodesic $\g\in \G$ is {\em forward recurrent}, if there exists a positive sequence $\{\tau_k\}\subset\Gam$ with $\tau_k\g\to\g$.

A subset $G\subset \G$ is {\em minimal}, if it is closed under the above notion of convergence and $\Gam$-invariant, i.e. $\tau\g\in G$ for all $\tau\in \Gam$ and $\g\in G$, and if $G$ contains no non-trivial, closed and $\Gam$-invariant subsets.

A minimal set $G\subset\G$ is {\em periodic}, if $G$ consists of all $\Gam$-translates of a single axis $\g\in\G$.
\end{defn}

Intuitively, a sequence $\{\tau_k\}$ is positive for $\g\in \G$, if $\{\tau_k\g\}$ describes the behavior of $\g(t)$ in the compact quotient $M$, as $t\to\infty$. Note that, if $G_0$ is a minimal closed and $\phi_g^t$-invariant subset of the $g$-unit tangent bundle of $M$, then the set of $g$-geodesics $G\subset \G$, which project into $G_0$, is minimal in the sense of Definition \ref{def positive gamma}. Also observe that, if $G$ is minimal, then for all $\g\in G$ there exists a positive (and a negative) sequence $\{\tau_k\}\subset\Gam$ for $\g$ with $\overline{\{\tau_k\g\}} = G$. In particular, every $\g\in G$ is bi-recurrent.

\begin{lemma}[upper semi-continuity of width]\label{width semi-cont}
Let $\g,\g'\in\G$ and $\{\tau_k\}\subset\Gamma$ be a positive sequence for $\g$, such that $\tau_k \g \to \g'$. Then $w_+(\g(\infty)) \leq w_0(\g')$.
\end{lemma}

The proof of Lemma \ref{width semi-cont} can be found in \cite{paper1}, Lemma 4.5. In particular, since $w_0(\g)$ is bounded for all $\g\in \G$ by a global constant $D(F,g)$ by the Morse Lemma, we obtain finiteness of all widths $w_+(\del)$ of $\del\in S^1$.

\begin{cor}\label{width of recurrent dir}
 If $\g$ is forward recurrent, then $w_+(\g(\infty))=w_0(\g)$. If $G\subset \G$ is minimal, then
 \[ w_\pm|_G = w_0|_G =: w(G) = \const. \]
\end{cor}

To prove Main Theorem \ref{main thm}, we need to show that
\[ w_+(\del)=0 \qquad \forall \del\in S^1-\Fix(\Gam). \]
Take some $\del\in S^1-\Fix(\Gam)$ and any $\g\in \G$ with $\g(\infty)=\del$. In the compact $g$-unit tangent bundle of $M$, the $\om$-limit set of $\g$ contains a minimal closed and $\phi_g^t$-invariant set. Let $G\subset \G$ be the minimal set of $g$-geodesics projecting into this minimal set. Lemma \ref{width semi-cont} and Corollary \ref{width of recurrent dir} show that
\[ w_+(\del) = w_+(\g(\infty)) \leq w(G), \]
so in order to prove Main Theorem \ref{main thm}, we will prove
\begin{align}\label{eq proves main thm}
 w(G)=0 \qquad \forall \text{ minimal, non-periodic } G\subset\G .
\end{align}

We will have to study the way that $g$-geodesics $\g\in G$ approximate other $g$-geodesics $\g'\in G$. For this, we use two lemmata, which can be found as Lemmata 4.6 and 4.7 in \cite{paper1}. 

\begin{lemma}\label{one-sided approx}
Let $\g,\g'\in \G$ and $\{\tau_k\}\subset\Gamma$ with $\tau_k \g\to \g'$, such that in $X$ we have $\tau_k \g\cap \g'=\emptyset ~\forall k$. Then $w_0(\g)=0$.
\end{lemma}

\begin{lemma}\label{periodic approx}
Let $\g,\g'\in \G$ and $\{\tau_k\}\subset\Gamma$ with $\tau_k \g\to \g'$, such that $\g'$ is an axis and $\tau_k\g\neq \g' ~  \forall k$. Then $w_0(\g)=0$.
\end{lemma}

Arguing by contradiction, we assume that for some fixed minimal and non-periodic set $G\subset\G$ we have
\[ w(G)>0. \]
In the following, we will use the following notion of convergence of minimal geodesics: A {\em sequence of minimal geodesics $c_k:\R\to X$ converges to a minimal geodesic $c:\R\to X$}, if there exists a sequence $\{T_k\}\subset \R$ with $\dot c_k(T_k) \to \dot c(0)$.

\begin{lemma}\label{lemma A}
 If there exists a non-periodic, minimal set $G\subset \G$ with $w(G)>0$, then there exists a bi-recurrent $\g\in \G$, $i\in \{0,1\}$ and a backward ray $c:(-\infty,0]\to X$ with $c(-\infty)=\g(-\infty), ~ c(0)\in c_\g^i(\R)$ and for $\{j\} = \{0,1\}-\{i\}$
 \[ \liminf_{t\to-\infty} d_g(c_\g^j(\R), c(t)) = 0, \qquad \liminf_{t\to-\infty} d_g(c_\g^i(\R), c(t)) = w_0(\g). \]
\end{lemma}

 \begin{figure}[!htb]\centering
 \includegraphics[scale=0.5]{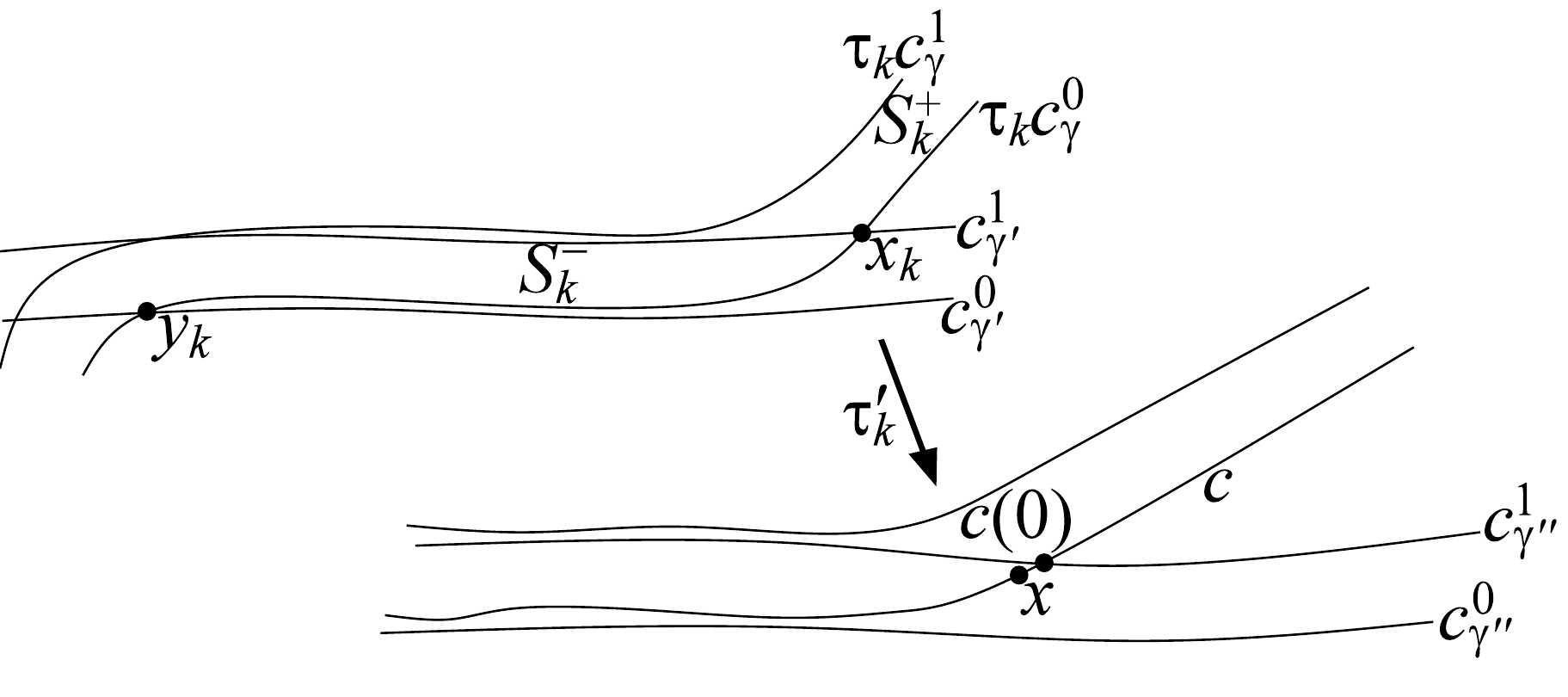}
 \caption{The objects in the proof of Lemma \ref{lemma A}. \label{fig_assumption_A}}
 \end{figure}

\begin{proof}
 As a consequence of Lemma \ref{periodic approx}, the set $G$ contains no axes. Lemma \ref{one-sided approx} then shows that, if $\g,\g'\in G$ and $\{\tau_k\}\subset \Gam$ positive for $\g$ with $\tau_k\g\to\g'$, then we have w.l.o.g.
\begin{align*}%\label{eq approx A}
 \tau_k\g\cap\g' \neq \emptyset \qquad \forall k.
\end{align*}
 Figure \ref{fig_assumption_A} depicts the following arguments. We can after passing to a subsequence assume that $\g_k(\pm\infty)>\g'(\pm\infty)$ for all $k$ in the counterclockwise orientation of $S^1$, the other case being analogous. We find a point of intersection $x_k$ of the two minimal geodesics $c_{\g'}^1(\R), \tau_k c_{\g}^0(\R)$ and choose a sequence $\{\tau_k'\}\subset\Gamma$, such that (after passing to subsequences) $\tau_k'\g'\to \g''\in G$ and such that $\{\tau_k'x_k\}$ converges to some $x\in X$. Let $y_k$ be a point of intersection of $c_{\g'}^0(\R), \tau_kc_{\g}^0(\R)$, then $d_g(x_k,y_k)\to\infty$ (any limit of $\{\tau_kc_{\g}^0\}$ lies in $\M(\g')$ and hence left of $c_{\g'}^0$). For $k\to\infty$, the sequence $\{\tau_k'\tau_kc_{\g}^0\}$ has a convergent subsequence with a limit minimal geodesic $c:\R\to X$, now connecting $\lim_k\tau_k'y_k = \g''(-\infty)$ to $\lim_k\tau_k'x_k = x$.

 We claim that the minimal geodesic $c(t)$ has to leave the strip between $c_{\g''}^i$ to the left, as $t\to+\infty$. To see this, observe that at each $x_k$, left of $c_{\g'}^1$, there is a strip $S_k^+$ of width $w(G)$ in the positive end of the strip between $\tau_kc_\g^i$, which will under $\{\tau_k'\}$ converge to a strip of width $w(G)$ left of $c_{\g''}^1$ starting at $x$. Since $w_+(\g''(\infty))=w(G)< 2 w(G)$ by $w(G)>0$, the limit of the strips $\tau_k'S_k^+$ cannot have asymptotic direction $\g''(\infty)$.

 Hence, we can assume that $c(0)\in c_{\g''}^1(\R)$. Moreover, the strips $S_k^-$ between $\tau_k c_\g^i$ left of $x_k$ will, under application of $\{\tau_k'\}$, converge to a strip in the negative end of the strip between $c_{\g''}^i$ (by $\tau_k'y_k \to \g''(-\infty)$ it has to lie in bounded distance of $\g''$ and by the same reasoning as above it cannot lie outside of the strip between $c_{\g''}^i$). Hence, we have
\[ \liminf_{t\to-\infty} d_g(c_{\g''}^0(\R), c(t)) = 0, \qquad \liminf_{t\to-\infty} d_g(c_{\g''}^1(\R), c(t)) = w(G) = w_0(\g''). \]
 The $g$-geodesic $\g''\in G$ is the $\g$ from the statement of the lemma. Bi-recurrence of $\g''$ follows from the minimality of $G$.
\end{proof}

For simplicity, we turn the picture around, replacing negative recurrence etc. by forward recurrence etc. (formally by considering the Finsler metric $v\mapsto F(-v)$).

\begin{defn}
 Two rays $c_0,c_1:[0,\infty)\to X$ with $c_i(\infty)=\del$ are {\em fully separated at $+\infty$}, if
 \[ \liminf_{t\to\infty} d_g(c_0[0,\infty), c_1(t)) = w_+(\del) . \]
 Analogously we explain when two rays in $\RR_-(\g(-\infty))$ is fully separated from each other at $-\infty$.
\end{defn}

In view of Lemma \ref{lemma A}, the proof of the following theorem will finish the proof of Main Theorem \ref{main thm}.

\begin{thm}\label{thm fully separated recurrent}
 If $\g\in \G$ is bi-recurrent with $w_0(\g)>0$ and if $c_*$ is any of the bounding geodesics $c_\g^0,c_\g^1$ of $\M(\g)$, then there cannot exist a ray $c$ in $\RR_+(\g(\infty))$ initiating from $c_*(\R)$ and being fully separated from $c_*$ at $+\infty$. 
\end{thm}

We fix in the following the bi-recurrent $g$-geodesic $\g\in\G$ and a positive sequence $\{\tau_k\}\subset \Gam$ for $\g$ with $\tau_k\g\to \g$. Let us moreover assume (using Corollary \ref{width of recurrent dir})
\[ w_0(\g) = w_+(\g(\infty))=w_-(\g(-\infty))>0 . \]

We prove three lemmata before turning to the proof of Theorem \ref{thm fully separated recurrent}.

\begin{lemma}\label{int dH = 1}
 Let $\del\in S^1$ and $u\in \H_+(\del)$. Then for all $v\in \RR_+(\del)$ the function
 \[ t\in [0,\infty) \mapsto  t- u\circ c_v(t) \]
 is bounded and monotonically increasing.
\end{lemma}

\begin{proof}
 We let $c,c_0:[0,\infty)\to X$ be two rays with $\dot c,\dot c_0\in \RR_+(\del)$. Recall that by property (1) in Lemma \ref{existence horofunction} and $u\in\H_+$, we have for $a\leq b$
 \[ a- u\circ c(a) \leq b-u\circ c(b) , \]
 hence the lower bound by $-u\circ c(0)$ and monotonicity are trivial. By the Morse Lemma, we find
 \[ B := \sup_{t\geq 0}d_g(c_0[0,\infty),c(t)) <\infty, \]
 and hence for all $t\geq 0$ we find some $s\in \R$ with $d_g(c_0(s),c(t))\leq B$. We first assume that $s\leq t$, then by minimality of $c,c_0$ and the triangle inequality
 \begin{align*}
   t & = d_F(c(0),c(t)) \\
   & \leq d_F(c(0),c_0(0)) + d_F(c_0(0),c_0(s)) + d_F(c_0(s),c(t)) \\
   & \leq d_F(c(0),c_0(0) + s + c_F\cdot B.
 \end{align*}
 If $t\leq s$, then similarly
 \begin{align*}
  s & = d_F(c_0(0),c_0(s)) \\
  & \leq d_F(c_0(0),c(0)) + d_F(c(0),c(t)) + d_F(c(t),c_0(s)) \\
  & \leq d_F(c_0(0),c(0)) + t + c_F\cdot B.
 \end{align*}
 In any case, $|t-s|$ is bounded from above by $d_F(c(0),c_0(0) + c_F B$ and hence
 \begin{align*}
  \frac{1}{c_F} \cdot d_g(c_0(t),c(t)) & \leq d_F(c_0(t),c(t)) \\
  & \leq d_F(c_0(t),c_0(s))+d_F(c_0(s),c(t)) \\
  & \leq |t-s|+c_F\cdot B \\
  & \leq d_F(c(0),c_0(0) + 2 \cdot c_F \cdot B =: C. 
 \end{align*}
 The fact that $u$ is Lipschitz with respect to $d_g$ with Lipschitz constant $c_F$ shows
 \begin{align*}
   |u\circ c(t) - u\circ c_0(t)|\leq c_F \cdot d_g(c_0(t), c(t) ) \leq c_F^2\cdot C.
 \end{align*}
 Assume now that $\dot c_0\in \J_+(u)$, then by definition $u\circ c_0(t) -t \equiv u \circ c_0(0)$ and hence
 \begin{align*}
  |u\circ c(t) - t| \leq |u\circ c(t) - u\circ c_0(t)| + |u\circ c_0(t) - t| \leq c_F^2 \cdot C + |u\circ c_0(0)|.
 \end{align*}
\end{proof}

The next lemma shows that, given the ``homotopy information'' $\{\tau_k\}$, that leads to forward recurrence of $\g$, we can also find forward recurrent motions in $\M(\g)$ with the same ``homotopy information''.

\begin{lemma}\label{existence recurrence}
 There exists a pair of minimal geodesics $c_0,c_1$ in $\M(\g)$ and sequences $T_k^0,T_k^1\to \infty$, such that $\dot c_0(0)$ is a limit point of $\{d\tau_k(\dot c_0(T_k^0))\}_k$ and $\dot c_1(0)$ is a limit point of $\{d\tau_k(\dot c_1(T_k^1))\}_k$, while $c_0$ and $c_1$ are fully separated at $-\infty$ and at $+\infty$.
\end{lemma}

Note that it will now be enough to prove Theorem \ref{thm fully separated recurrent} for $c_*\in\{c_0,c_1\}$ given by Lemma \ref{existence recurrence}, since any ray initiating from one of the bounding geodesics will cross the corresponding $c_i$ by full separation.

\begin{proof}
 Let us say that a subset $A\subset \M(\g)$ is $\{\tau_k\}$-invariant, if for all $v$ in $A$ any limit geodesic of $\{\tau_k c_v\}$ belongs again to $A$. It follows from $w_0(\g)>0$, that the following two subsets are $\{\tau_k\}$-invariant:
 \begin{align*}
  A_0 &:= \big\{ v \in \M(\g) : \inf_{t\in\R} d_g(c_\g^1(\R),c_v(t)) \geq w_0(\g) \big\}, \\
  A_1 &:= \big\{ v \in \M(\g) : \inf_{t\in\R} d_g(c_\g^0(\R),c_v(t)) \geq w_0(\g) \big\} .
 \end{align*}
 Obviously, both sets $A_0,A_1$ are closed sets and any pair of minimal geodesics $c_0\in A_0,c_1\in A_1$ is fully separated at $\pm\infty$. Moreover, the bounding geodesics $c_\g^i$ lie in $A_i$, i.e. $A_i\neq\emptyset$. By full separation, one can also easily show that both sets $A_i$ define a lamination of $X$ (any two geodesics from $A_0$, say, come close at $\pm\infty$, then use Lemma \ref{crossing minimals}).
 
 We will argue for the set $A_0$, the case of $A_1$ being the same. We wish to find a ``minimal subset'' $\hat A_0$ of $A_0$. For this, let $\mathfrak{A}_0$ be the collection of non-empty, closed, $\{\tau_k\}$-invariant subsets of $A_0$, then $\mathfrak{A}_0\neq\emptyset$ by $A_0\in \mathfrak{A}_0$ and $\mathfrak{A}_0$ is partially ordered by $\subset$. Moreover, any linearly ordered subset of $\mathfrak{A}_0$ has the intersection of its members as a $\subset$-smallest element (the intersection is non-empty by Cantor's Intersection Theorem, if we think of the lamination $A_0$ as a set of points in some compact, transverse interval). Hence, Zorn's Lemma shows the existence of the desired $\hat A_0$, which has thus the property of not containing any non-trivial, closed, $\{\tau_k\}$-invariant subsets.
 
 Since $A_0$ defines a lamination of $X$ and $\hat A_0$ is closed, there exists an uppermost minimal geodesic $c_0$ in $\hat A_0$, and we claim that $c_0$ is $\{\tau_k\}$-recurrent. Namely, let $\hat c_0$ be the uppermost limit geodesic of $\{\tau_k c_0\}$, then $\hat c_0$ belongs to $\hat A_0$, so does not lie above $c_0$. But by minimality of $\hat A_0$, the subset of $\hat A_0$ consisting of all the $\{\tau_k\}$-limit geodesics of $\hat A_0$ equals $\hat A_0$ and hence $c_0$ is the limit of some $c$ in $\hat A_0$. By the preservation of the ordering in the foliation in $A_0$ under $\{\tau_k\}$-limits, $c_0$ is also the limit of itself: the limiting process of $\tau_k c\to c_0$ pushes also the limiting process of $\tau_k c_0\to \hat c_0$ up to $c_0$.
\end{proof}

Let $c_*:\R\to X$ be one of the minimal geodesics and $\{T_k^*\}$ be the corresponding sequence given by Lemma \ref{existence recurrence}. We replace $\{\tau_k\}$ by a subsequence in order to obtain $d\tau_k(\dot c_*(T_k^*)) \to \dot c_*(0)$. Moreover, we let $c:[0,\infty)\to X$ be a ray with $c(\infty)=c_*(\infty)$, such that $c,c_*$ are fully separated at $+\infty$. The idea behind the following lemma is that $c$ has limits, which have the same average displacement along the ``homotopy path'' $\{\tau_k\}$ as $c_*$.

\begin{lemma}\label{cor sim rec}
 With the above notation, there exists a sequence $S_k\to\infty$, such that
 \[ \liminf_{k \to \infty} d_g(c(S_k),\tau_k c(T_k^*+S_k)) = 0. \]
\end{lemma}

\begin{proof}
 Let $c_1$ in $\M(\g)$ be a limit geodesic of $\{\tau_k c\}$ and after taking a further subsequence, we may assume $\tau_k c \to c_1$. Since $c$ is fully separated from $c_*$, so will be $c_1$, such that
 \begin{align*}%\label{eq c,c_1 close}
  \liminf_{S\to\infty} d_g(c_1(\R),c(S)) = 0.
 \end{align*}
 We fix $\e>0$ and let $S \geq 0$, such that $d_g(c_1(\R),c(S)) \leq \e$. Since $\tau_kc\to c_1$, we then find some $k_0\in \N$, such that (using positivity of $\{\tau_k\}$ for $c_*$)
 \begin{align}\label{eq c almost rec}
  \exists T_k\to\infty : \qquad d_g( c(S) , \tau_k c(T_k+S) ) \leq 2\e \quad \forall k\geq k_0 . 
 \end{align}
 Moreover, by Lemma \ref{int dH = 1}, we may also assume that $S$ is large enough, such that
 \[ \left| \big[ (S+T)-u\circ c(S+T) \big] - \big[ S-u\circ c(S) \big] \right| \leq \e \quad \forall T\geq 0. \]
 and hence
 \begin{align}\label{eq c almost cal}
 \left| T_k-u\circ c(S+T_k) + u\circ c(S) \right| \leq \e .  
 \end{align}
 We let $u=\lim\tau_k u \in \H_+(c_*(\infty))$ (cf. the discussion after Lemma \ref{bounding horofunctions} for the existence of such $u$). Let us see that
 \begin{align}\label{recurrent u-cal}
  T = u\circ c_*(T)-u\circ c_*(0) \qquad \forall T\geq 0.
 \end{align}
 Indeed, by $\tau_ku\to u$ and $\tau_k c_*(. + T_k^*) \to c_*$ (both in $C_{loc}^0$) we find
  \begin{align*}
  T - u\circ c_*|_0^T & = \lim_{k\to\infty} T - (\tau_ku)\circ \tau_k c_*(. + T_k^*)|_0^T \\
  & = \lim_{k\to\infty} T - u\circ \tau_k^{-1}\circ \tau_k c_*(. + T_k^*)|_0^T \\
  & = \lim_{k\to\infty} T - u \circ c_*(. + T_k^*)|_0^T \\
  & = 0
 \end{align*}
 by Lemma \ref{int dH = 1}. Now observe, that for $k$ sufficiently large we have
 \begin{align*}
  &~ |T_k-T_k^*| \\
  \leq &~ | T_k-u\circ c(S+T_k) + u\circ c(S)| + | u\circ c(S+T_k) - u\circ c(S) - T_k^*| \\
  \stackrel{\eqref{eq c almost cal},\eqref{recurrent u-cal}}{\leq} &~ \e + | u\circ c(S+T_k) - u\circ c(S) - u\circ c_*(T_k^*) + u\circ c_*(0)| \\
  \leq &~ \e + | u\circ c(S+T_k) - u\circ \tau_k^{-1} c(S) | \\
  &~ \qquad + | u\circ \tau_k^{-1} c(S) - u\circ c(S) -u\circ \tau_k^{-1}c_*(0) + u\circ c_*(0)| \\
  &~ \qquad + | u\circ \tau_k^{-1}c_*(0)- u\circ c_*(T_k^*) | \\
  \leq &~ \e + c_F\cdot d_g(c(S+T_k) , \tau_k^{-1} c(S) ) \\
  &~ \qquad + | (\tau_ku)\circ c(S)  - u\circ c(S) -(\tau_ku)\circ c_*(0) + u\circ c_*(0)| \\
  &~ \qquad + c_F\cdot d_g (\tau_k^{-1}c_*(0) , c_*(T_k^*) ) \\
  \stackrel{\eqref{eq c almost rec}}{\leq} &~ \e + c_F\cdot 2\e \\
  &~ \qquad + | (\tau_ku)\circ c(S)  - u\circ c(S)| + |(\tau_ku)\circ c_*(0) - u\circ c_*(0)| \\
  &~ \qquad + c_F\cdot d_g (c_*(0) , \tau_k c_*(T_k^*) ) .
 \end{align*}
 By $\tau_ku\to u$ and $\tau_k c_*(. + T_k^*) \to c_*$, we find for large $k$, that
 \[ |T_k-T_k^*| \leq  2\e (1+c_F) , \]
 so that \eqref{eq c almost rec} and the triangle inequality show
 \begin{align*}
  &~ d_g( c(S) , \tau_k c(T_k^*+S) ) \\
  \leq &~ d_g( c(S) , \tau_k c(T_k+S) ) + d_g( \tau_k c(T_k+S), \tau_k c(T_k^*+S) ) \\
  \leq &~ 2\e + |T_k-T_k^* | \\
  \leq &~ 2\e + 2\e (1+c_F)
 \end{align*}
 for $k$ sufficiently large and $S$ as chosen above.
\end{proof}

\begin{figure}[!htb]\centering
\includegraphics[scale=0.35]{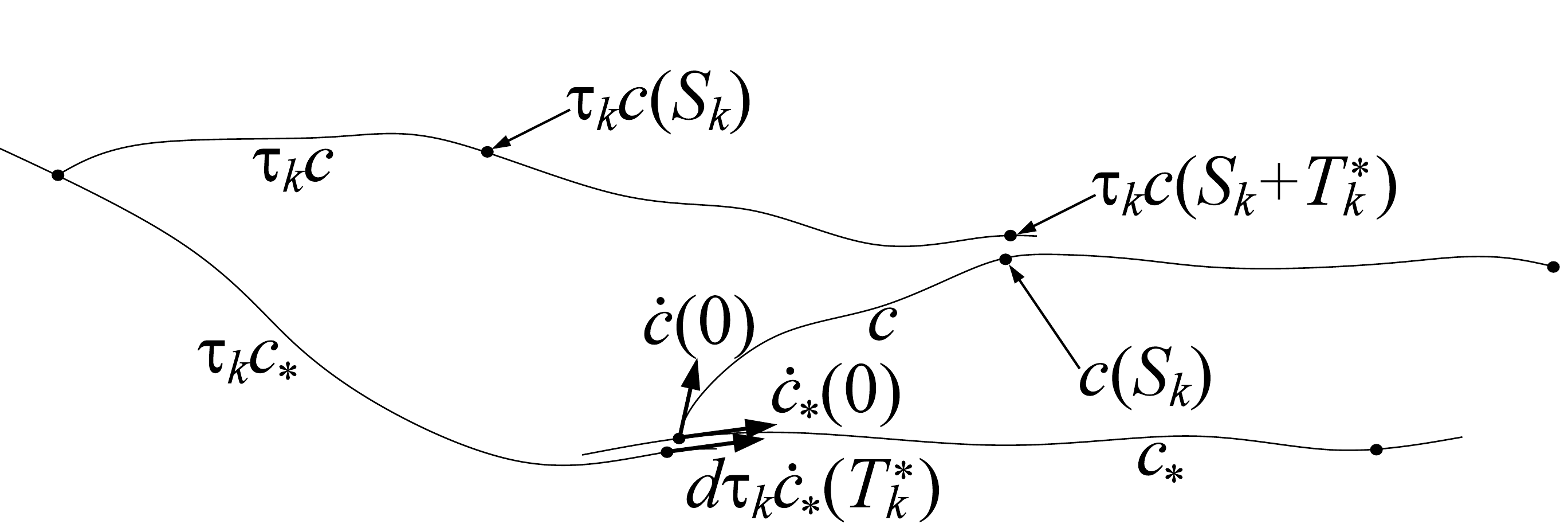}
\caption{The argument in the Proof of Theorem \ref{thm fully separated recurrent}. \label{fig morse-arg2}}
\end{figure}

\begin{proof}[Proof of Theorem \ref{thm fully separated recurrent}]
 Let $c_*$ and $\{T_k^*\}$ be given by Lemma \ref{existence recurrence}, such that after taking subsequences we have $d\tau_k(\dot c(T_k^*))\to \dot c_*(0)$. Moreover, let as above $c:[0,\infty)\to X$ be a ray with $c(0)=c_*(0)$ and $c(\infty)=c_*(\infty)$. Assuming $c$ to be fully separated from $c_*$ at $+\infty$, in particular $\dot c(0)\neq \dot c_*(0)$, we have to arrive at a contradiction. We saw in Lemma \ref{cor sim rec}, that we find a sequence $S_k\to\infty$ with $\liminf_{k \to \infty} d_g(c(S_k),\tau_k c(T_k^*+S_k)) = 0$ and after passing to another subsequence, we suppose
 \begin{align}\label{c_1-rec}
  d_F(c(S_k),\tau_k c(T_k^*+S_k)) \to 0 , \qquad k \to \infty.
 \end{align}
 By $\dot c(0)\neq \dot c_*(0)$ we find a small $\delta>0$ and, depending on the angle between $\dot c_*(0)$ and $\dot c(0)$, some $\e>0$ with
 \begin{align*}
  d_F(c_*(-\delta),c(\delta)) \leq 2\delta-\e.
 \end{align*}
 By $d\tau_k(\dot c_*(T_k^*))\to \dot c_*(0)$ we then obtain for large $k$, that
 \begin{align}\label{shorten kink}
  d_F(\tau_k c_*(T_k^*-\delta),c(\delta)) \leq 2\delta-\e/2.
 \end{align}
 We find by minimality of $\tau_k c$ and $c_*(0)=c(0)$, that for sufficiently large $k$
 \begin{align*}
  T_k^* + S_k & = d_F(\tau_k c(0),\tau_k c(S_k+T_k^*)) \\
  & \leq d_F(\tau_k c_*(0),\tau_kc_*(T_k^*-\delta))+ d_F(\tau_k c_*(T_k^*-\delta),c(\delta))  \\
      & \qquad + d_F(c(\delta), c(S_k)) + d_F(c(S_k) ,\tau_kc(S_k+T_k^*)) \\
  & \stackrel{\eqref{c_1-rec}, \eqref{shorten kink}}{\leq} d_F(\tau_k c_*(0),\tau_k c_*(T_k^*-\delta))+ 2\delta-\e/2 \\
      & \qquad + d_F(c(\delta), c(S_k)) + \e/4 \\
  & = T_k^* + S_k -\e/4 .
 \end{align*}
 This is a contradiction.
\end{proof}

\subsection{The proofs of the corollaries}\label{subsection cor}

\begin{proof}[Proof of Corollary \ref{cor no intersections}]
 The fact that rays with common endpoint $\del\notin \Fix(\Gam)$ can intersect only in a common initial point can be deduced directly from $w_+(\del) = 0$ and Lemma \ref{crossing minimals}.
\end{proof}

\begin{proof}[Proof of Corollary \ref{cor local entropy}]
 Use the arguments in the proof of Lemma 4.5 in \cite{GKOS}.
\end{proof}

\begin{proof}[Proof of Corollary \ref{cor unique busemann}]
 Let $\del\notin \Fix(\Gam)$ and $u\in \H_+(\del)$. If $v\in \RR_+(\del)$ and $\e>0$, then by Corollary \ref{cor no intersections}, the ray $c_v:[\e,\infty)\to X$ belongs to $\J_+(u)$. The closedness of $\J_+(u)$ shows $v\in \J_+(u)$, i.e. $\RR_+(\del)\subset \J_+(u)$ and using Corollary \ref{as dir well-def} we have $\RR_+(\del) = \J_+(u)$ for all $u\in \H_+(\del)$. Hence Corollary \ref{J=J' then u=u'} shows the uniqueness of the horofunction $u\in \H_+(\del)$.
\end{proof}

\begin{proof}[Proof of Corollary \ref{R Lipschitz graph}]
 The fact that $\phi_F^\e\RR_+(\del)$ with $\del\notin \Fix(\Gam)$ and $\e>0$ is locally a Lipschitz graph follows directly from Lemma \ref{lemma fathi} and $\RR_+(\del)=\J_+(u)$ in Corollary \ref{cor unique busemann}.
\end{proof}

\begin{proof}[Proof of Corollary \ref{cor unique min geod}]
 Let $\del\notin \Fix(\Gam)$ and, using the bounding geodesics of $\M(\g)$ from Lemma \ref{bounding geodesics morse}, set
\begin{align*}
 L &:= \bigcup \{ c_\g^i(\R) : \g\in \G, \g(\infty)=\del , i=0,1\} , \\
 A &:= \{ C\subset X : C \text{ connected component of $X-L$}\}. 
\end{align*}
$L$ defines a closed lamination of $X$ by Corollary \ref{cor no intersections} (for closedness recall the construction of the $c_\g^i$) and, moreover, for any $\g\in \G$ with $\g(\infty)=\del$ and $c_\g^0(\R) \neq c_\g^1(\R)$, the (connected) strip $C_\g\subset X-L$ between the $c_\g^i(\R)$ is an element of $A$. The claim follows, since $A$ is countable: $X$ is the union of countably many compact sets $K_n$, and each $K_n$ can contain at most countably many disjoint open sets.

For $\del\in\Fix(\Gam)$, the set $L$ above also defines a lamination of $X$, which can be seen directly from Theorem \ref{morse periodic} and Lemma \ref{crossing minimals}. Hence, the above reasoning applies to all $\del\in S^1$.
\end{proof}

\begin{proof}[Proof of Corollary \ref{cor backwards dir}]
 Fixing $\del\in S^1$, consider the lamination $L$ of $X$ as in the proof of Corollary \ref{cor unique min geod}. If $x\in L$, then $x$ lies on some $c_\g^i$, which determines $\g$ uniquely. If $x\in X-L$, then $x$ lies in some open strip $C_\g$, i.e. between $c_\g^0$ and $c_\g^1$, again determining $\g$ uniquely.
\end{proof}

\begin{proof}[Proof of Corollary \ref{cor generic}]
 Let $\del\in \Fix(\Gam)$. Observe that, if $\M_{per}(\g)$ contains only one minimal geodesic, then by Theorem \ref{morse periodic}, all rays in $\RR_+(\del)$ are in fact asymptotic. Hence, Main Theorem \ref{main thm} holds for $\del\in \Fix(\Gam)$ and so do its corollaries.
\end{proof}

\pagebreak


\begin{thebibliography}{99}

\bibitem[Ban88]{bangert} V. Bangert -- \emph{Mather sets for twist maps and geodesics on tori}. Dynamics Reported 1 (1988), 1-56.

\bibitem[BCS00]{BCS} D. D.-W. Bao, S. S. Chern, Z. Shen -- \emph{An introduction to Riemann-Finsler geometry}. Graduate Texts in Mathematics 200, Springer Verlag (2000).

\bibitem[CIPP98]{contreras1} G. Contreras, R. Iturriaga, G. P. Paternain, M. Paternain -- \emph{Lagrangian graphs, minimizing measures and \mane's critical values}. Geometric and Functional Analysis 8 (1998), 788-809.

\bibitem[CS14]{coudene-schapira} Y. Coud\'ene, B. Schapira -- \emph{Generic measures for geodesic flows in non-positively curved manifolds}. Preprint (2014).

%\bibitem[Fat08]{fathi} A. Fathi -- \emph{Weak KAM theorem in Lagrangian dynamics, preliminary version number 10}. Preprint (2008), available online at \url{http://www.math.u-bordeaux1.fr/~pthieull/Recherche/KamFaible/publications.html}.

\bibitem[GKOS14]{GKOS} E. Glasmachers, G. Knieper, C. Ogouyandjou, J. P. Schr\"oder -- {\em Topological entropy of minimal geodesics and volume growth on surfaces}. Journal of Modern Dynamics 8.1 (2014), 75-91.

%\bibitem[HM42]{morse_hedlund} G. A. Hedlund, H. M. Morse -- \emph{Manifolds without conjugate points}. Transactions of the American Mathematical Society 51 (1942), 362-386.

\bibitem[Hed32]{hedlund} G. A. Hedlund -- \emph{Geodesics on a two-dimensional Riemannian manifold with periodic coefficients}. Annals of Mathematics 33.4 (1932), 719-739.

\bibitem[KH96]{KH} A. Katok, B. Hasselblatt -- \emph{Introduction to the modern theory of dynamical systems}. Encyclopedia of Mathematics and its Applications 54, Cambridge University Press (1996).

\bibitem[Kli71]{klingenberg} W. Klingenberg -- \emph{Geod\"atischer Flu{\ss} auf Mannigfaltigkeiten vom hyperbolischen Typ}. Inventiones Mathematicae 14 (1971), 63-82.

%\bibitem[KOS13]{minimal_preprint} G. Knieper; C. Ogouyandjou; J. P. Schr\"oder -- \emph{Topological entropy of minimal geodesics and volume growth on surfaces}. arXiv:1308.2127 [math.DG] (2013).

\bibitem[Mor24]{morse} H. M. Morse -- \emph{A fundamental class of geodesics on any closed surface of genus greater than one}. Transactions of the American Mathematical Society 26.1 (1924), 25-60.

\bibitem[Sch14a]{paper1_diss} J. P. Schr\"oder -- \emph{Global minimizers for Tonelli Lagrangians on the 2-torus}. Preprint (2014), available online at \url{http://www.ruhr-uni-bochum.de/ffm/Lehrstuehle/Lehrstuhl-X/jan.html}.

\bibitem[Sch14b]{paper1} J. P. Schr\"oder -- \emph{Minimal rays on surfaces of genus greater than one}. arXiv:1404.0573v1 [math.DS] (2014).

\bibitem[Sch14c]{paper_generic} J. P. Schr\"oder -- \emph{Uniqueness of shortest closed geodesics for generic Finsler metrics}. arXiv:1405.2756v1 [math.DG] (2014).

%\bibitem[Sor10]{sorrentino} A. Sorrentino -- \emph{Lecture notes on Mather's theory for Lagrangian systems}. arXiv: 1011.0590 [math.DS] (2010).

%\bibitem[Wal00]{walters} P. Walters -- \emph{An introduction to ergodic theory}. Graduate Texts in Mathematics 79, Springer Verlag (2000).

\bibitem[Zau62]{zaustinsky} E. M. Zaustinsky -- \emph{Extremals on compact {\it E}-surfaces}. Transactions of the American Mathematical Society 102.3 (1962), 433-445.

\end{thebibliography}
\end{document}